\documentclass{amsart}
\usepackage{times,mathrsfs,array,amssymb,pb-diagram,multirow, color, hyperref}

\newcounter{lemma}
\newtheorem{Theorem}{Theorem}
\newtheorem{Lemma}[lemma]{Lemma}

\newtheorem{Proposition}[lemma]{Proposition}
\newtheorem{theorem}{Theorem}

\theoremstyle{definition}
\newtheorem{Example}[lemma]{Example}

\newtheorem{Notation}[lemma]{Notation}

\newtheorem{Remark}[lemma]{Remark}

\def\H{\mathbb H}

\def\Emb{\mathrm{Emb}}

\def\tr{\operatorname{tr}}
\def\trd{\operatorname{trd}}
\def\nrd{\operatorname{nrd}}

\def\g{\gamma}
\def\G{\Gamma}
\def\B{\mathcal B}

\def\F{\mathbb F}

\def\N{\mathbb N}
\def\O{\mathcal O}

\def\Q{\mathbb Q}
\def\R{\mathbb R}
\def\Z{\mathbb Z}
\def\be{\boldsymbol e}
\def\bone{\boldsymbol 1}
\def\new{\mathrm{new}}

\def\mod{\  \mathrm{mod}\ }

\def\gen#1{\langle #1\rangle}
\def\JS#1#2{\left(\frac{#1}{#2}\right)}

\def\SL{\mathrm{SL}}

\def\CM{\mathrm{CM}}

\def\SM#1#2#3#4{\left(\begin{smallmatrix}#1&#2\\#3&#4\end{smallmatrix}
	\right)}

\newcommand{\tabcaption}{\def\@captype{table}\caption}

\begin{document}
	
	\title{Jacquet-Langlands correspondence for non-Eichler orders}
	
	\author{Fang-Ting Tu}
	\address{Department of Mathematics, Louisiana State University, Baton
		Rouge, LA70803-4918, USA}
	\email{tu@math.lsu.edu}
	
	\author{Yifan Yang}
	\address{Department of Mathematics, National Taiwan
		University and National Center for Theoretical Sciences, Taipei,
		Taiwan 10617}
	\email{yangyifan@ntu.edu.tw}
	\date{\today}
	\subjclass[2000]{}
	\thanks{The authors would like to thank John Voight for his helpful comments on this paper. The first author is partially supported by the NSF grant DMS \#2302531; the second author was partially supported by Grant 113-2115-M-002 -003 -MY3   of the Ministry of Science and Technology, Taiwan (R.O.C.).}
	
	\begin{abstract} 
		In this note, we give a concrete realization of the Jacquet-Langlands correspondence for non-Eichler orders of indefinite quaternion algebras defined over $\Q$. To be more precise, we consider a special type of index-two suborder of the Eichler order of level $N$ in the quaternion algebra with an even discriminant $D$. 
	\end{abstract}
	\maketitle
	
	\section{Introduction and statements of results}
	
	Let $\B$ be an indefinite quaternion algebra of discriminant $D$ over
	$\Q$. Up to conjugation, there is a unique embedding $\iota$ from  $\B$ into $M(2,\R)$. Let $\O$ be an order in $\B$  and $\O^1$ be its norm-one group.  Then $\G(\O):=\iota(\O^1)$ is a discrete subgroup $\SL(2,\R)$ of first kind and acts on the upper half plane $\H$ via the linear fractional transformations. When $\B\neq M(2,\Q)$, we let $X(\O)$ be the compact Riemann surface $\G(\O)\backslash\H$.  To ease notation, we will also use $\G(\O)$ to indicate the norm one group $\O^1$ in $\B$.

	For a positive squarefree integer $D$ with an even number of
	prime divisors and a positive integer $N$ relatively prime to $D$,
	we let $\O(D,N)$ denote the Eichler order of level $N$ in the
	quaternion algebra $\B_D$ of discriminant $D$ over $\Q$
	and $\Gamma(D,N)$ be its norm-one group.   For a positive even integer $k$, let $S_k(\Gamma(D,N))$ be the space of modular forms of weight $k$ on $\Gamma(D,N)$. Also, for a positive integer $M$, let $S_k(\Gamma_0(M))$ denote the space of modular forms of weight $k$ on $\Gamma_0(M)$. Then the classical Jacquet-Langlands correspondence for $\O(D,N)$ can be stated in the following form.
	
	\begin{theorem}[\cite{Hida, JL}]
		Let $D>1$ be a positive squarefree integer with an even number of prime divisors and $N$ be a positive integer relatively prime to $D$. For a positive integer $n$ relatively prime to $DN$, let $T_n$ denote the Hecke operator on $S_k(\Gamma(D,N))$ or $S_k(\Gamma_0(DN))$. We have 
		$$
		\tr(T_n|S_k(\Gamma(D,N)))=\tr(T_n|S_k(\Gamma_0(DN)^{D\text{-new}})).
		$$
	\end{theorem}

	Here for a positive integer $L$ and a positive divisor $M$ of $L$, we let
	$$
	S_k(\Gamma_0(L))^{M\text{-new}}:=
	\bigoplus_{M'|L,M|M'}\{g(m\tau):g\in S_k(\Gamma_0(M'))^{\text{new}},
	m|(L/M')\},
	$$
	where $ S_k(\Gamma_0(M))^{\text{new}}$ denotes the newform subspace of $S_k(\Gamma_0(M))$. 
	
	A natural question to ask is whether analogous correspondences exist for non-Eichler orders of $\B_D$.
	In view of automorphic representations, such correspondences for non-Eichler orders exist. However, Hecke operators in the case of non-Eichler orders may not have a clean description as in the case of Eichler orders. Also,
	it is hard to match spaces of modular forms on a non-Eichler order to spaces of classical modular forms on some congruence subgroup of $\SL(2,\Z)$. As far as we know, there are few explicit realizations of Jacquet-Langlands correspondence for non-Eichler orders in an indefinite quaternion algebra known in literature (see \cite[et al.]{Gelbart, Gross, Shimizu} for local consideration, and  \cite[et al.]{HPS-basis, Martin, Ponomarev, Voight-book} for definite quaternion algebra cases). The purpose of this paper is to provide such an example.
	
	We first describe the non-Eichler orders we are interested in. Throughout the paper, we assume that $D$ is even. In this case,
	the function $w:\alpha\to\frac12v_2(\nrd(\alpha))$ defines a discrete
	valuation on the division algebra $\B_D\otimes_\Q\Q_2$, where $v_2$ is the $2$-adic valuation and $\nrd(\alpha)$ denotes the reduced norm of $\alpha$. Then the
	maximal $\Z_2$-order $R=\O(D,N)\otimes_\Z\Z_2$ is equal to the
	valuation ring of $\B_D\otimes_\Q\Q_2$ with respect to $w$. Let $P$ be
	the unique maximal (two-sided) ideal of $R$. We have $R/P\simeq\F_4$
	(see \cite[Theorem 13.1.6]{Voight-book}). Then $\Z_2+P$ is a suborder
	of index $2$ of $R$. 
	It follows that, by the local-global correspondence for orders in $\B_D$, the Eichler order $\O(D,N)$ has a suborder $\O'(D,N)$ of index $2$ such
	that $\O'(D,N)\otimes_\Z\Z_p=\O(D,N)\otimes_\Z\Z_p$ for odd prime $p$
	and $\O'(D,N)\otimes_\Z\Z_2$ has index $2$ in $\O(D,N)\otimes_\Z\Z_2$. The following is an explicit example of such an order.
	
	\begin{Example} Let $D=2p_1\ldots p_r$, where $p_1,\ldots,p_r$ are all
		congruent to $3$ modulo $4$ and $r$ is odd. Then $\O(D,1)$ and
		$\O'(D,1)$ can be realized as
		$$
		\O(D,1)=\Z+\Z i+\Z j+\Z\frac{1+i+j+ij}2, \qquad
		\O'(D,1)=\Z+\Z i+\Z j+\Z ij,
		$$
		where $i^2=-1$ and $j^2=p_1\ldots p_r$. In the case of $D=6$, the group $\G(\O(6,1))/\{\pm 1\}$ is generated by the elements 
		$$
		\gamma_2=i, \quad \g_3=(1-3i+j-k)/2, \quad\g_4=(1-3i-j-k)/2,   \quad  \g_1=(\g_2\g_3\g_4)^{-1}=2i+j,
		$$
		where $\g_1^2=\g_2^2=\g_3^3=\g_4^3=-1$. The  group $\G(\O'(6,1))/\{\pm 1\}$  is generated by 
		$$
		\gamma_3^m\gamma_1\gamma_3^{-m} \mbox{and} \quad \gamma_3^m\gamma_2\gamma_3^{-m}, \quad m=0, 1, 2. 
		$$
	\end{Example}
	
	\medskip
	
	We let $\Gamma(D,N)$ and $\Gamma'(D,N)$ denote the groups of norm-one
	elements in $\O(D,N)$ and $\O'(D,N)$, and $X(D,N)$ and $X'(D,N)$
	denote the Shimura curves associated to 
	$\O(D,N)$ and $O'(D,N)$, respectively. Also, we let $S_k(\Gamma(D,N))$
	and $S_k(\Gamma'(D,N))$ be the spaces of modular forms of weight
	$k$ on $\Gamma(D,N)$ and $\Gamma'(D,N)$, respectively.
	On the space $S_k(\Gamma'(D,N))$, we can define Hecke operators in the same way as $S_k(\Gamma(D,N))$.
	Namely, for a positive integer $n$ relatively prime to $DN$, let $M(n)$ be the set of elements of reduced norm $n$ in $\O'(D,N)$ (by Lemma \ref{lemma: Gamma'(D,N)}(4), $M(n)$ is nonempty). Then the Hecke operator $T_n$ on $S_k(\Gamma'(D,N))$ can be defined by
	$$
	T_n:f\longmapsto n^{k/2-1}
	\sum_{\gamma\in\Gamma'(D,N)\backslash M(n)}f\big|_k\gamma.
	$$
	We have the following Jacquet-Langlands correspondence for the non-Eichler order
	$\O'(D,N)$.
	
	\begin{Theorem} \label{theorem: JL}
		With $D$ and $N$ given as above, we have, for all
		positive integer $n$ such that $(n,DN)=1$ and all positive even
		integers $k$,
		$$
		\tr(T_n\big|S_k(\Gamma'(D,N)))
		=\tr(T_n\big|S_k(\Gamma_0(DN))^{D\text{-new}})
		+2\tr(T_n\big|S_k(\Gamma_0(2DN))^{2D\text{-new}}).
		$$
	\end{Theorem}
	
	Theorem \ref{theorem: JL} can be refined as follows.
	The group $\Gamma'(D,N)$ is a
	normal subgroup of index $3$ of $\Gamma(D,N)$ (see Lemma \ref{lemma:
		Gamma'(D,N)}). Thus, $\Gamma(D,N)$ acts on the space
	$S_k(\Gamma'(D,N))$.
	For a character $\chi$ of the quotient group
	$\Gamma(D,N)/\Gamma'(D,N)$, we let
	$$
	S_k(\Gamma'(D,N),\chi):=\{f\in S_k(\Gamma'(D,N)):
	f\big|_k\alpha=\chi(\alpha)f\text{ for all }\alpha\in
	\Gamma(D,N)\}.
	$$
	Thus, we have a direct sum decomposition
	$$
	S_k(\Gamma'(D,N))=\bigoplus_\chi S_k(\Gamma'(D,N),\chi).
	$$
	It is easy to see that this is an orthogonal decomposition
	with respect to the Petersson inner product and each summand is invariant under Hecke operators.
	When $\chi$ is the trivial character $\chi_0$,
	$S_k(\Gamma'(D,N),\chi_0)$ is the same as $S_k(\Gamma(D,N))$.
	In view of the classical Jacquet-Langlands correspondence for Eichler orders, it is
	natural to guess that 
	\begin{equation} \label{eq: JL1}
		\tr(T_n|S_k(\Gamma'(D,N),\chi))=
		\tr(T_n|S_k(\Gamma_0(2DN))^{2D\text{-new}}
	\end{equation}
	for a nontrivial character $\chi$.
	The next theorem shows that this is indeed the case.
	
	\begin{Theorem} \label{theorem: JL1}
		Let $D$ and $N$ be as above. Let $\chi$ be a nontrivial
		character of the group $\Gamma(D,N)/\Gamma'(D,N)$. Then \eqref{eq:
			JL1} holds for all positive integers $n$ with
		$(n,DN)=1$. Equivalently, for all 
		positive even integers $k$, the two spaces
		$S_k(\Gamma'(D,N),\chi)$ and $S_k(\Gamma_0(2DN))^{2D\text{-new}}$
		are isomorphic as Hecke modules.
	\end{Theorem}

	\section{Preliminaries}
	
	\subsection{Optimal embeddings and CM-points}
	\label{subsection: CM-points}
	
	Since the trace formulas involve CM-points, we briefly review the
	notion of CM-points and formulas for the number of CM-points on a
	modular curve or a Shimura curve in this section.
	
	Let $\B$ be an indefinite quaternion algebra of discriminant $D$ over
	$\Q$ and $\O$ be an order in $\B$. We fix an embedding
	$\iota$ of $\B$ into $M(2,\R)$. In order for a quadratic number
	field $K$ to be embeddable into $\B$, the necessary and
	sufficient condition is $\JS{K}p\neq1$ for any prime divisor $p$ of
	$D$, where $\JS{K}p$ is the Kronecker symbol. Now suppose that $K$ can
	be embedded into $\B$, say, $\sigma:K\hookrightarrow\B$ is an
	embedding. Then $\sigma(K)\cap\O=\sigma(R)$ for some quadratic order
	$R$ in $K$. Let $d$ be the discriminant of $R$. Then we say $\sigma$
	is an \emph{optimal embedding} of discriminant $d$ into $\O$. We let
	$\Emb(d;\O)$ denote the set of all such embeddings. Note that if
	$\sigma\in\Emb(d;\O)$, then $\gamma\sigma\gamma^{-1}$ also
	belongs to $\Emb(d;\O)$ for $\gamma\in\O^1$, where $\O^1$
	denotes the group of norm-one elements in $\O$.
	
	Now if $K$ is an imaginary quadratic number field, then
	$\iota(\sigma(K))$ has a common fixed point $\tau_\sigma$ on
	$\H$. This point is called a \emph{CM-point} of discriminant $d$.
	It is clear that for $\gamma\in\O^1$, we have
	$\tau_{\gamma\sigma\gamma^{-1}}=\iota(\gamma)\tau_\sigma$. Thus, each
	conjugacy class in $\mathrm{Emb}(d;\O)$ by $\O^1$ determines a unique point on
	$X(\O)$.
	
	\begin{Notation} For a modular curve or a Shimura curve $X$ and a
		negative discriminant $d$, we let $\CM(d;X)$ denote the set of
		CM-points of discriminant $d$ on $X$.
	\end{Notation}
	
	Note, however, that the correspondence between $\Emb(d;\O)/\O^1$ and
	$\CM(d;X(\O))$ is not one-to-one. This is due to the fact that
	if $\sigma\in\Emb(d;\O)$, then $\overline\sigma: K\hookrightarrow\B$ defined by
	$\overline\sigma(a):=\sigma(\overline a)$ is also an optimal embedding
	of discriminant $d$ with the same fixed point
	$\tau_{\overline\sigma}=\tau_\sigma$.
	To get a one-to-one correspondence, we consider the $(2,1)$-entry of
	$\iota(\sigma(\sqrt d))$. A simple computation shows that the
	$(2,1)$-entries of $\iota(\sigma(\sqrt d))$ are either all positive or
	all negative for $\sigma$ in a given conjugacy class of optimal
	embeddings. We say $\sigma$ is \emph{positive} (respectively,
	\emph{negative}) and write $\sigma>0$ (respectively, $\sigma<0$) if
	the $(2,1)$-entry of $\iota(\sigma(\sqrt d))$ is positive
	(respectively, negative). (In some literature, a positive embedding is
	called normalized instead.) We let
	$\Emb^+(d;\O):=\{\sigma\in\Emb(d;\O):\sigma>0\}$. Then the
	correspondence between $\Emb^+(d;\O)/\O^1$ and $\CM(d;X(\O))$ is
	one-to-one.
	
	The determination of the cardinality of $\CM(d;X(\O))$ is usually
	done locally. For a prime $p$, we let $\O_p:=\O\otimes_\Z\Z_p$ and
	similarly let $\Emb(d;\O_p)$ denote the set of optimal embeddings of
	discriminant $d$ into $\O_p$.
	
	\begin{Lemma}[{\cite[Theorem 30.7.3]{Voight-book}}]
		\label{lemma: |CM| 0}
		With the notations given as above, let
		$e(d;\O_p)=|\Emb(d;\O_p)/\O_p^\times|$. Then
		$$
		|\CM(d;X(\O))|=h(d)\prod_p e(d;\O_p),
		$$
		where $h(d)$ is the class number of the order of discriminant $d$ in $K$.
	\end{Lemma}
	
	Define the Eichler symbol $\left\{\frac
	dp\right\}$ by 
	$$
	\left\{\frac dp\right\}=\begin{cases}
		\JS{d_0}p, &\text{if }p\nmid f, \\
		1, &\text{if }p|f, \end{cases}
	$$
	where $\JS{d_0}p$ is the Kronecker symbol. We now record formulas for $e(d;\O_p)$ relevant to our discussion.
	
	\begin{Lemma}[{\cite[Proposition 5, Chapter II]{Eichler}}]
		\label{lemma: |CM| 1}
		Let $\O(D,N)$ be an Eichler order of level $N$ in the indefinite
		quaternion algebra of discriminant $D$ over $\Q$ ($D=1$ allowed).
		Given a negative discriminant $d$, we write $d$ as
		$d=f^2d_0$, where $d_0$ is a fundamental discriminant and $f$ is a
		positive integer. 
		\begin{enumerate}
			\item If $p|D$, then $e(d;\O(D,N)_p)=1-\left\{\frac dp\right\}$.
			\item If $p\|N$, then $e(d;\O(D,N)_p)=1+\left\{\frac dp\right\}$.
		\end{enumerate}
	\end{Lemma}
	
	The case $p^2|N$ is more complicated. For our purpose, we only need
	the formula for the case $p=2$ and $4\|N$.
	
	\begin{Lemma}[{\cite[Theorem 2]{Ogg}}] \label{lemma: |CM| 2}
		Let $M$ be an odd positive integer. Given a negative
		discriminant $d$, write $d$ as $d=f^2d_0$, where $d_0$ is a
		fundamental discriminant and $f$ is a positive integer. Then
		$$
		e(d;\O(1,4M)_2)=\begin{cases}
			0, &\text{if }4|d_0,2\nmid f, \\
			3, &\text{if }4|d_0,2|f, \\
			1+\JS{d_0}2, &\text{if }d_0\equiv1\mod 4, 2\nmid f, \\
			3+\JS{d_0}2, &\text{if }d_0\equiv1\mod 4, 2\|f, \\
			3, &\text{if }d_0\equiv1\mod 4, 4|f. \end{cases}
		$$
	\end{Lemma}
	
	\subsection{The Shimura curve $X'(D,N)$}
	
	In this section, we collect some properties of the group
	$\Gamma'(D,N)$ and the Shimura curve $X'(D,N)$ that will be
	needed later on. 
	\begin{Lemma}\label{lemma: Gamma'(D,N)}
		\begin{enumerate}
			\item An element $\alpha$ of $\O(D,N)$ is contained in $\O'(D,N)$
			if and only if the reduced trace $\trd(\alpha)$ is even. Also, an element $\alpha$
			of $\Gamma(D,N)$ is contained in $\Gamma'(D,N)$ if and only if
			$\nrd(\alpha-1)$ is even.
			\item We have $\Gamma'(D,N)\lhd\Gamma(D,N)$ and
			$[\Gamma(D,N):\Gamma'(D,N)]=3$.
			\item Let $\O(D,N)^\times$ be the unit group of $\O(D,N)$. Then $\Gamma'(D,N)\lhd\O(D,N)^\times$ and $\O(D,N)^\times/\Gamma'(D,N)$ is cyclic of order $6$.
			\item The reduced norm map $\nrd:\O'(D,N)\to\Z$ is surjective.
			\item Let $p$ be a prime not dividing $DN$. Suppose that $\gamma_1$ and $\gamma_2$ are two elements of reduced norm $p$ in $\O'(D,N)$. Then there are elements $\alpha$ and $\beta$ in $\Gamma(D,N)$ such that $\gamma_1=\alpha\gamma_2\beta$ and $\alpha\beta\in\Gamma'(D,N)$.
		\end{enumerate}
	\end{Lemma}
	
	\begin{proof} Recall that the unique division quaternion algebra over
		$\Q_2$ can be realized as $\JS{-1,-1}{\Q_2}$. The unique maximal
		$\Z_2$-order of $\JS{-1,-1}{\Q_2}$ is
		$R:=\Z_2+\Z_2i+\Z_2j+\Z_2(1+i+j+ij)/2$. Its maximal ideal is
		$P=(i+j)=\{a_0+a_1i+a_2j+a_3ij:a_m\in\Z_2,a_0+\cdots+a_3\text{ is
			even}\}$ and the suborder $\Z_2+P$ of index $2$ in the maximal
		order is $R':=\Z_2+\Z_2i+\Z_2j+\Z_2ij$. Therefore, we
		have $\B_D\otimes_\Q\Q_2\simeq\JS{-1,-1}{\Q_2}$, and
		the images of $\O(D,N)\otimes_\Z\Z_2$ and $\O'(D,N)\otimes_\Z\Z_2$
		under the isomorphism are $R$ and $R'$, respectively. From this, we
		immediately see that an element $\alpha$ of $\O(D,N)$ is contained
		in $\O'(D,N)$ if and only if $\trd(\alpha)$ is even. 
		
		Now suppose $\alpha\in\Gamma(D,N)$. We have
		$\nrd(\alpha-1)=\nrd(\alpha)-\trd(\alpha)+1=2-\trd(\alpha)$. By the
		characterization of elements of $\Gamma'(D,N)$ above,
		we see that $\alpha$ is in $\Gamma'(D,N)$ if and only if
		$\nrd(\alpha-1)$ is even. 
		
		We now prove Part (2). We regard $\O(D,N)$ as a subring of $R$. Observe that $R/P\simeq\F_4$ and the only elements $a$ of $\F_4$ such that $\operatorname{tr}_{\F_4/\F_2}(a)=0$ are those elements in $\F_2$. Consequently, by Part (1), an element $\gamma$ of $\O(D,N)$ is in $\O'(D,N)$ if and only if $\gamma\equiv0,1\mod P$. In particular,
		an element $\gamma$ of $\Gamma(D,N)$ is in $\Gamma'(D,N)$ if and only if
		$\gamma\equiv 1\mod P$. In other words, $\Gamma'(D,N)$ is the kernel
		of the reduction homomorphism 
		$$\Gamma(D,N)\longrightarrow(R/P)^\times \quad \mbox{defined
			by} \quad  \alpha\mapsto\alpha\mod P.
		$$ Thus, $\Gamma'(D,N)\lhd\Gamma(D,N)$
		and the index of $\Gamma'(D,N)$ in $\Gamma(D,N)$ is either $1$ or
		$3$. In view of Part (1), we only need to show that $\Gamma(D,N)$
		has an element of odd trace.
		
		Recall that, as a consequence of the strong approximation
		for Eichler orders in an indefinite quaternion algebra over $\Q$,
		the reduction map $\Gamma(D,N)\mapsto(\O(D,N)/2\O(D,N))^1$ is
		surjective (see Theorem 28.2.11 of \cite{Voight-book}), where
		$(\O(D,N)/2\O(D,N))^1$ denotes the group of elements
		$\alpha+2\O(D,N)$ such that $\nrd(\alpha)\equiv 1\mod 2$. Now it is
		easy to see that the embedding $(\O(D,N)/2\O(D,N))^1\hookrightarrow
		(R/2R)^1$ is actually an isomorphism. Since $R$ has an element
		$(1+i+j+ij)/2$ of norm $1$ and trace $1$, we see that $\Gamma(D,N)$
		has an element of odd trace. This completes the proof of the lemma.

		To prove Part (3), we first note that, by Part (2), $[\O(D,N)^\times:\Gamma'(D,N)]=[\O(D,N)^\times:\Gamma(D,N)][\Gamma(D,N):\Gamma'(D,N)]=6$. Moreover, using the characterization of elements of $\O'(D,N)$ given in Part (1), we easily see that $\Gamma'(D,N)$ and $\O'(D,N)^\times$ are both normal subgroups of $\O(D,N)^\times$. Now the proof of Part (2) can also be used to show that $\O'(D,N)^\times$ is a subgroup of $\O(D,N)^\times$ of index $3$.
		Thus, $\O(D,N)^\times/\Gamma'(D,N)$ is a group of order $6$ having
		and a normal subgroup $\O'(D,N)/\Gamma'(D,N)$ of index $3$. Therefore, $\O(D,N)^\times/\Gamma'(D,N)$ is cyclic of order $6$.
		
		We next prove Part (4).
		By the strong approximation theorem for the Eichler order $\O(D,N)$, the reduced norm map $\nrd:\O(D,N)\to\Z$ is surjective for $\O(D,N)$.
		For an integer $n$, let $\gamma$ be an element of reduced norm $n$ in $\O(D,N)$. Let $\alpha$ be an element of $\Gamma(D,N)$ not in $\Gamma'(D,N)$. Then $1$, $\alpha$, and $\overline\alpha$ form a complete set of coset representatives of $\Gamma'(D,N)$ in $\Gamma(D,N)$. By Part (1), $\trd(\alpha)$ is odd.
		Thus, $1+\alpha+\overline\alpha$ is an even integer.
		It follows that $\tr(\gamma+\alpha\gamma+\overline\alpha\gamma)$ is an even integer. Consequently, at least one of $\gamma$, $\alpha\gamma$, and $\overline\alpha\gamma$ has an even trace. This element of even trace is an element of reduced norm $n$ in $\O'(D,N)$, by Part (1) again.
		
		We now prove Part (5). Since $\gamma_1$ and $\gamma_2$ are both elements of reduced norm $p$ in the Eichler order $\O(D,N)$, by the strong approximation theorem, there exist elements $\alpha$ and $\beta$ in $\Gamma(D,N)$ such that $\gamma_1=\alpha\gamma_2\beta$. To prove that $\alpha\beta\in\Gamma'(D,N)$, we regard $\O(D,N)$ as a subring of $R$ as in the proof of Part (2). Then the images of $\gamma_1$ and $\gamma_2$ under the reduction homomorphism $R\to R/P\simeq\F_4$ are both $1$. Therefore,
		the image of $\alpha\beta$ under the homomorphism is also $1$. Consequently, $\alpha\beta\in\Gamma'(DN)$. This completes the proof of the lemma.
	\end{proof}
	
	\begin{Lemma} \label{lemma: CM on X'}
		\begin{enumerate}
			\item The covering $X'(D,N)\to X(D,N)$ has degree $3$.
			The branch points of the covering are exactly the elliptic points
			of order $3$ (if such elliptic points exist).
			\item Let $d$ be a negative discriminant. If $d\equiv1\mod 4$, then
			$$
			|\CM(d;X'(D,N))|=0.
			$$
			If $d\equiv0\mod 4$, then
			$$
			|\CM(d;X'(D,N))|=\begin{cases}
				|\CM(-3;X(D,N))|, &\text{if }d=-12, \\
				3|\CM(d/4;X(D,N))|, &d/4\equiv1\mod 4\text{ and }d\neq-12,\\
				3|\CM(d;X(D,N))|, &\text{else}.\end{cases}
			$$
		\end{enumerate}
	\end{Lemma}
	
	\begin{proof}
		The assertion that $X'(D,N)\to X(D,N)$ has degree $3$ follows from
		Lemma \ref{lemma: Gamma'(D,N)}(1). The branch points of the covering
		can only occur possibly at elliptic points of $X(D,N)$.
		To determine which elliptic points are branch points, we use the
		result in Part (2), which we prove now.

		Let $\phi:K\hookrightarrow\B_D$ be an
		embedding of imaginary quadratic number field $K$ into $\B_D$.
		Let $d_1$ and $d_2$ be the discriminants of $\phi$ as an 
		optimal embedding into $\O(D,N)$ and $\O'(D,N)$, respectively.
		Let us analyze the relation between $d_1$ and $d_2$.
		
		If $d_1\equiv0\mod 4$, then
		$\phi(K)\cap\O(D,N)=\phi(\Z[\sqrt{d_1}/2])$. Since every element in
		$\phi(\Z[\sqrt{d_1}/2])$ has an even trace, by Lemma \ref{lemma:
			Gamma'(D,N)}, $\phi(K)\cap\O'(D,N)$ is equal to
		$\phi(\Z[\sqrt{d_1}/2])$. Thus, $d_2=d_1$ when $d_1\equiv0\mod4$.
		If $d_1\equiv1\mod 4$, then $\phi(K)\cap\O(\Z[(1+\sqrt{d_1})/2])
		=\phi(\Z[\sqrt{d_1}])$. Thus, $d_2=4d_1$ when $d_1\equiv0\mod 4$.
		
		The discussion above shows that every point on $X'(D,N)$ that is
		mapped to a CM-point of discriminant $d_1$ on $X(D,N)$ in the
		covering 
		$X'(D,N)\to X(D,N)$ is a CM-point of discriminant
		$$
		\begin{cases}
			4d_1, &\text{if }d_1\equiv1\mod 4,\\
			d_1, &\text{if }d_1\equiv 0\mod 4. \end{cases}
		$$
		This in particular shows that elliptic points of order $3$ (i.e.,
		CM-points of discriminant $-3$) on $X(D,N)$ are branch
		points of the covering $X'(D,N)\to X(D,N)$, and elliptic points of
		order $2$ (i.e., CM-points of discriminant $-4$) on $X(D,N)$ are not
		branch points.
		
		Furthermore, observe that if $d_1$ is an odd discriminant, then by
		Lemma \ref{lemma: |CM| 1}, the set $\CM(4d_1;X(D,N))$ is
		empty. Therefore, every CM-point of discriminant $4d_1$ on $X'(D,N)$
		must lie in the preimage of some CM-point
		of discriminant $d_1$ on $X(D,N)$. In other words, we have
		$|\CM(4d_1;X'(D,N))|=3|\CM(d_1;X(D,N))|$, except when $d_1=-3$, in
		which case we have $|\CM(-12;X'(D,N))|=|\CM(-3;X(D,N))|$ instead. 
		Likewise, if $d_1\equiv 0\mod 4$, then every CM-point of
		discriminant $d_1$ on $X'(D,N)$ lies in the preimage of some
		CM-points of the same discriminant on $X(D,N)$. Therefore, we have
		$|\CM(d_1;X'(D,N))|=3|\CM(d_1;X(D,N))|$. This completes the proof.
	\end{proof}
	
	\section{Proof of Theorem \ref{theorem: JL}}
	To prove Theorem \ref{theorem: JL}, we will compare the trace formulas on both sides of the identity. The key fomulas are list as Propoistion \ref{proposition: trace formula} and Lemma \ref{lemma: convolution} below.
	Throughout the section, we let $D$ and $N$ be given
	in the statement of Theorem \ref{theorem: JL}. For a positive integer $n$ relatively prime to $DN$, we let
	$$
	M(n):=\{\gamma\in\O'(D,N)):\nrd(\gamma)=n\}.
	$$
	
	By Lemma \ref{lemma: Gamma'(D,N)}(4), $M(n)$ is nonempty.
	
	The trace formulas for modular forms in the  setting of Shimura curves in literature  are all about modular forms on Eichler orders. Since $\O'(D,N)$ is not an Eichler order,  here we briefly sketch the proof of the proposition (although the proof is very similar to the case of Eichler orders).
	
	\begin{Proposition}[Hecke trace formula for $\Gamma'(D,N)$)] \label{proposition: trace formula}
		We have
		\begin{equation} \label{eq: preliminary formula}
			\begin{split}
				\tr(T_n\big| S_k(\Gamma'(D,N)))
				&=\frac{k-1}{4}\alpha_nn^{k/2-1}\phi(D)\psi(N) \\
				&-\frac12\sum_{\substack{t\in\Z\\t^2<4n}}
				\frac{\rho_{t,n}^{k-1}-\overline\rho_{t,n}^{k-1}}
				{\rho_{t,n}-\overline\rho_{t,n}}
				\sum_{r^2d=t^2-4n}\frac1{w_d}|\CM(d;X'(D,N))| \\
				&+\beta_k\sum_{t|n}t,
			\end{split}
		\end{equation}
		where $\phi$ is the Euler totient function,
		\begin{equation} \label{eq: psi}
			\psi(N)=N\prod_{p|N}\left(1+\frac1p\right),
		\end{equation}
		\begin{equation} \label{eq: wd}
			\alpha_n=\begin{cases}
				1, &\text{if }n\text{ is a square}, \\
				0, &\text{else}, \end{cases} \qquad
			w_d=\begin{cases}
				2, &\text{if }d=-4, \\
				3, &\text{if }d=-3, \\
				1, &\text{else}, \end{cases} \qquad
		\end{equation}
		\begin{equation} \label{eq: beta}
			\beta_k=\begin{cases}
				1, &\text{if }k=2, \\
				0, &\text{else},
			\end{cases}
		\end{equation}
		and $\rho_{t,n}=(t+\sqrt{t^2-4n})/2$ denotes the root of the
		polynomial $x^2-tx+n$ with a positive imaginary part.
	\end{Proposition}
	
	\begin{proof}   Here we adopt the approach of Zagier \cite{Zagier}. Fix an embedding
		$\iota:\B_D\to M(2,\R)$.
		Then Theorem 1 of \cite{Zagier}, adapted to our setting, states that
		$$
		\tr(T_n\big|S_k(\Gamma'(D,N)))
		=A_kn^{k-1}\sum_{\gamma\in M(n)}I_\gamma,
		$$
		where
		$$
		A_k=\frac{(-1)^{k/2}2^{k-3}(k-1)}\pi,
		$$
		and
		$$
		I_\gamma:=\iint_{F}
		\sum_{\gamma\in M(n)}\frac{y^k}{(c|\tau|^2+d\overline\tau
			-a\tau-b)^k}\frac{dx\,dy}{y^2}.
		$$
		Here $a,b,c,d$ are the entries in $\iota(\gamma)=\SM abcd$,
		$\tau=x+iy$, and the integral is over a fundamental domain $F$ of
		$\iota(\Gamma'(D,N))$ in $\H$.
		
		Partition the sum according to the trace $t$ of $\gamma$ and
		write
		$$
		\tr(T_n\big|S_k(\Gamma'(D,N)))
		=A_kn^{k-1}\sum_{t\in\Z}I(t), \qquad
		I(t):=\sum_{\substack{\gamma\in M(n),\trd(\gamma)=t}}I_\gamma.
		$$
		Consider the cases $t^2-4n=0$, $t^2-4n<0$, and $t^2-4n>0$
		separately, which correspond to the actions on $\H$ given by  parabolic, elliptic, and hyperbolic elements of $\SL(2,\R)$, respectively, if $\gamma \neq \pm I$. 
		
		The case $t^2-4n=0$ occurs only when $n$ is a square.
		In such a case, each of
		$I(\pm 2\sqrt n)$ consisting of one single term
		$$
		I_{\pm\sqrt n}=\iint_F\frac{y^k}{(2i\sqrt ny)^k}\frac{dx\,dy}{y^2}
		=\frac{(-1)^{k/2}}{2^kn^{k/2}}\iint_{F}\frac{dx\,dy}{y^2}.
		$$
		Then by Lemma \ref{lemma: CM on X'},
		$$
		I_{\pm\sqrt n}
		=\frac{3(-1)^{k/2}}{2^kn^{k/2}}\iint_{\iota(\Gamma(D,N))
			\backslash\H}\frac{dx\,dy}{y^2}.
		$$
		According to \cite[Theorem 39.1.13]{Voight-book}, the last integral
		is equal to $\frac\pi3\phi(D)\psi(N)$. Thus, the total contribution
		from the case $t^2-4n=0$ to the trace is
		\begin{equation} \label{eq: trace term 1}
			\begin{split}
				\begin{cases}
					0, &\text{if }n\text{ is not a square}, \\
					\frac{k-1}4n^{k/2-1}\phi(D)\psi(N),
					&\text{if }n\text{ is a square}. \end{cases}
			\end{split}
		\end{equation}
		
		For the case $t^2-4n<0$, we shall show that
		\begin{equation} \label{eq: trace term 2}
			\begin{split}
				A_kn^{k-1}I(t)=-\frac12
				\frac{\rho_{t,n}^{k-1}-\overline\rho_{t,n}^{k-1}}
				{\rho_{t,n}-\overline\rho_{t,n}}
				\sum_{r^2d=t^2-4n}\frac1{w_d}|\CM(d;X'(D,N))|.
			\end{split}
		\end{equation}
		Let $\Gamma'(D,N)$ act on $M(n)$ by conjugation. For
		$\gamma\in M(n)$, we let $\Gamma_\gamma$ denote the isotropy
		subgroup for $\gamma$. Also, given a conjugacy class $C$, we let
		$w_C=|\Gamma_\gamma/\pm1|$, where $\gamma$ is any element in $C$.
		We partition the sum $I(t)$
		according to conjugacy classes and write
		$$
		I(t)=\sum_C\sum_{\gamma\in C}I_\gamma,
		$$
		where the outer sum runs through all conjugacy classes $C$ contained
		in the set $\{\gamma\in M(n):\trd(\gamma)=t\}$. We can check that
		the $(2,1)$-entries of $\iota(\gamma)$ are either all positive or
		all negative for $\gamma\in C$.
		For convenience, we write $C>0$ (respectively, $C<0$) if the
		$(2,1)$-entries are positive (respectively, negative). Now
		following the computation in \cite{Zagier}, we can show that if
		$C>0$, then
		\begin{equation*}
			\begin{split}
				\sum_{\gamma\in C}I_\gamma
				&=\frac1{w_C}\iint_\H
				\frac{y^k}{(|\tau|^2-ity-(t^2/4-n))^k}\frac{dx\,dy}{y^2}\\
				&=\frac{(-1)^{k/2}\pi}{2^{k-2}(k-1)n^{k-1}w_C}
				\frac{\overline\rho_{t,n}^{k-1}}{\rho_{t,n}-\overline\rho_{t,n}},
			\end{split}
		\end{equation*}
		and if $C<0$, then
		\begin{equation*}
			\begin{split}
				\sum_{\gamma\in C}I_\gamma
				&=\frac1{w_C}\iint_\H
				\frac{y^k}{(|\tau|^2+ity-(t^2/4-n))^k}\frac{dx\,dy}{y^2}\\
				&=\frac{(-1)^{k/2}\pi}{2^{k-2}(k-1)n^{k-1}w_C}
				\frac{\overline\rho_{-t,n}^{k-1}}
				{\rho_{-t,n}-\overline\rho_{-t,n}}
				=\frac{(-1)^{k/2}\pi}{2^{k-2}(k-1)n^{k-1}w_C}
				\frac{(-\rho_{t,n})^{k-1}}
				{-\overline\rho_{t,n}+\rho_{t,n}}.
			\end{split}
		\end{equation*}
		Observe that if $C$ is a conjugacy class whose elements have trace
		$t$, then $\overline C:=\{\overline\gamma:\gamma\in C\}$ is also
		such a conjugacy class, where $\overline\gamma$ is the quaternionic
		conjugate of $\gamma$. Moreover, if $C>0$, then $\overline C<0$.
		Thus,
		\begin{equation} \label{eq: I(t) tmp}
			I(t)=\sum_{C>0}\left(\sum_{\gamma\in C}I_\gamma
			+\sum_{\gamma\in\overline C}I_\gamma\right)
			=-\frac{(-1)^{k/2}\pi}{2^{k-2}(k-1)n^{k-1}}
			\frac{\rho_{t,n}^{k-1}-\overline\rho_{t,n}^{k-1}}
			{\rho_{t,n}-\overline\rho_{t,n}}\sum_{C>0}\frac1{w_C}.
		\end{equation}
		Now each conjugacy class $C$ defines an equivalence class of
		embeddings $\sigma$ of $K:=\Q(\sqrt{t^2-4n})$ into $\B_D$ defined by
		$\sigma:r+s\rho_{t,n}\mapsto r+s\gamma$, where $\gamma$ is an
		element in $C$. The common fixed point of $\sigma(K)$ is a CM-point
		of discriminant $d$ on $X'(D,N)$ for some $d$ and $r$ satisfying
		$r^2d=t^2-4n$. Conversely, given a CM-point of discriminant $d$ on
		$X'(D,N)$ such that $r^2d=t^2-4n$ for some integer $r$, there
		corresponds an embedding $\sigma:K\hookrightarrow\B_D$ such that
		$\sigma(K)\cap\O'(D,N)=\sigma(R)$, where $R$ is the quadratic
		order of discriminant $d$ in $K$. Then $\gamma=(t+r\sigma(\sqrt
		d))/2$ is an element in $\O'(D,N)$ of trace $t$ and norm $n$.
		Changing $\gamma$ to $(t-r\sigma(\sqrt d))/2$ if necessary, we may
		assume that the conjugacy class of $\gamma$ is positive.
		Therefore, the set of positive conjugacy classes of trace $t$ and
		norm $n$ is in one-to-one correspondence with the set
		$\cup_{r^2d=t^2-4n}\CM(d;X'(D,N))$. Moreover, if $C$ is a conjugacy
		class corresponding to a CM-point of discriminant $-4$, then
		$w_C=2$; otherwise, $w_C=1$. (By Lemma \ref{lemma: CM on X'},
		$\CM(-3;X'(D,N))$ is empty.) Therefore, the sum $\sum_C1/w_C$ in
		\eqref{eq: I(t) tmp} can be written as
		$$
		\sum_{r^2d=t^2-4n}\frac1{w_d}|\CM(d;X'(D,N))|.
		$$
		Plugging this into \eqref{eq: I(t) tmp}, we obtain \eqref{eq: trace
			term 2}.
		
		Finally, the proof in \cite{Zagier} shows that the contribution of
		the terms with $t^2-4n>0$ is $0$. (We remark in \cite{Zagier} the
		case $t^2-4n=u^2$ for some $u\in\mathbb N$ needs to be considered
		separately. Here since $\B_D$ is a division algebra,
		$\trd(\gamma)^2-4\nrd(\gamma)$ cannot be a square for any
		$\gamma\in\B_D^\times$.) This completes the proof of
		\eqref{eq: preliminary formula}.
	\end{proof}
	
	\begin{Lemma}[{\cite[Theorem 12.4.11]{Cohen-Stromberg}}]
		\label{lemma: trace Gamma0(M)}
		Let $M$ be a positive integer. Then for a positive integer $n$
		relatively prime to $M$ and a positive even integer $k$, we have
		\begin{equation*}
			\begin{split}
				\tr(T_n\big|S_k(\Gamma_0(M)))
				&=\frac{k-1}{12}\psi(M)\alpha_n \\
				&\qquad-\frac12\sum_{\substack{t\in\Z\\t^2<4n}}
				\frac{\rho_{t,n}^{k-1}-\overline\rho_{t,n}^{k-1}}
				{\rho_{t,n}-\overline\rho_{t,n}}
				\sum_{r^2d=t^2-4n}\frac1{w_d}|\CM(d;X_0(M))| \\
				&\qquad-{\sum_{\substack{d|n\\d\le\sqrt n}}}'d^{k-1}
				\sum_{\substack{c|M\\(c,M/c)|(M,n/d-d)}}\phi((c,M/c))
				+\beta_k\sum_{t|n}t,
			\end{split}
		\end{equation*}
		where $\alpha_n$ and $w_d$ are defined in \eqref{eq: wd},
		$\beta_k$ is defined by \eqref{eq: beta},
		$\rho_{t,n}$ denotes a root of the polynomial $x^2-tx+n$, and
		$\displaystyle{\sum}'$ means that the term $d=n^{1/2}$, if present,
		is counted with coefficient $1/2$.
	\end{Lemma}
	
	\begin{Remark}
		For our purposes, we express the contribution of the case $t^2-4n$  in a different
		form than in \cite{Cohen-Stromberg}, see \cite{HPS-orders, HPS-newforms}, \cite[Section 4.2]{QuaternionBook} and \cite[Section 30.7]{Voight-book} for example.  
	\end{Remark}

	The proof of Theorem \ref{theorem: JL} will use properties of certain
	arithmetic functions, which we recall now. For two arithmetic functions $f$ and $g$ defined on
	$\N$, we let the (multiplicative) convolution $f\ast g$ be defined by
	$$
	(f\ast g)(n)=\sum_{d|n}f(d)g(n/d).
	$$
	Then the function $\be$ defined on $\N$ by
	$$
	\be(n)=\begin{cases}1, &\text{if }n=1, \\
		0, &\text{else}, \end{cases}
	$$
	is the identity element for this binary operation. We let
	$\sigma_0(n)=\sum_{d|n}\ell^0=\sum_{d|n}1$ be the divisor
	function. Note that the Dirichlet series of $\sigma_0(m)$ is
	$\zeta(s)^2$. Thus, if we let $\delta$ be the multiplicative function 
	that takes values
	$$
	\delta(p^e)=\begin{cases}
		-2, &\text{if }e=1, \\
		1, &\text{if }e=2, \\
		0, &\text{if }e\ge 3,\end{cases}
	$$
	at prime powers, then
	\begin{equation} \label{eq: delta}
		\sigma_0\ast\delta=\delta\ast\sigma_0=\be.
	\end{equation}
	Thus, if $f(n)$ and $g(n)$ are related by $f=\sigma_0\ast g$, i.e., if
	$$
	f(n)=\sum_{d|n}\sigma_0(d)g(n/d),
	$$
	then we have conversely, $g=(\delta\ast\sigma_0)\ast g=\delta\ast f$,
	i.e.,
	\begin{equation} \label{eq: inversion}
		g(n)=\sum_{d|n}\delta(d)f(n/d).
	\end{equation}
	In the next lemma we compute some sums involving $\sigma_0$ and
	$\delta$.
	
	\begin{Lemma} \label{lemma: convolution}
		\begin{enumerate}
			\item Let $M$ and $n$ be positive integers such that $(n,M)=1$.
			Then we have
			\begin{equation} \label{eq: tr(Tn) 1}
				\tr(T_n\big|S_k(\Gamma_0(M)))
				=\sum_{d|M}\sigma_0(M/d)\tr(T_n\big|S_k(\Gamma_0(d))^\new)
			\end{equation}
			and
			\begin{equation} \label{eq: tr(Tn) 2}
				\tr(T_n\big|S_k(\Gamma_0(M))^\new)
				=\sum_{d|M}\delta(M/d)\tr(T_n\big|S_k(\Gamma_0(d))).
			\end{equation}
			\item We have
			$$
			\sum_{d|M}\delta(d)=\mu(M)=\begin{cases}
				(-1)^r, &\text{if }M\text{ is a product of }r
				\text{ distinct prime}, \\
				0, &\text{else}. \end{cases}
			$$
			\item
			Let $\psi$ be the function defined by \eqref{eq: psi}.
			Then
			\begin{equation} \label{eq: psi convolution}
				\sum_{m|D}\delta(D/m)\psi(mN)=\phi(D)\psi(N).
			\end{equation}
			Also, for a negative discriminant $d$, we let
			\begin{equation} \label{eq: epsilon d}
				e_p(d)=e(d;\O(D,N)_p)
				=\begin{cases}
					1-\left\{\frac dp\right\}, &\text{if }p|D, \\
					e(d;\O(1,N)_p), &\text{if }p|N, \end{cases}
			\end{equation}
			where $e(d;\O_p)$ is defined as in Lemma \ref{lemma: |CM| 0}.
			Then we have
			\begin{equation} \label{eq: CM convolution}
				\sum_{m|D}\delta(D/m)|\CM(d;X_0(mN)|
				=h(d)\prod_{p|DN}e_p(d) =|\CM(d;X(D,N))|.
			\end{equation}
			\item We have
			\begin{equation} \label{eq: psi convolution 2}
				\sum_{m|2D}\delta(2D/m)\psi(mN)=\phi(D)\psi(N).
			\end{equation}
			Moreover, for a negative discriminant $d$, we write $d$ as
			$d=f^2d_0$, where $d_0$ is a fundamental discriminant and $f$ is a
			positive integer. We have
			\begin{equation} \label{eq: CM convolution 2}
				\sum_{m|2D}\delta(2D/m)|\CM(d;X_0(mN))|=h(d)
				\widetilde e_2(p)\prod_{p|DN,p\neq2}e_p(d),
			\end{equation}
			where
			\begin{equation} \label{eq: e'}
				\widetilde e_2(d)=\begin{cases}
					1, &\text{if }d_0\equiv 0\mod 4\text{ and }2\nmid f, \\  
					0, &\text{if }d_0\equiv 0\mod 4\text{ and }2|f, \\
					\JS{d_0}2, &\text{if }d_0\equiv 1\mod 4
					\text{ and }2\nmid f, \\
					-\JS{d_0}2, &\text{if }d_0\equiv 1\mod 4\text{ and }2\|f,\\
					0, &\text{if }d_0\equiv 1\mod 4\text{ and }4|f. \end{cases}
			\end{equation}
			and $e_p(d)$ are defined by \eqref{eq: epsilon d}.
		\end{enumerate}
	\end{Lemma}
	
	\begin{proof}
		We recall that the space $S_k(\Gamma_0(M))$ has an orthogonal
		decomposition
		$$
		S_k(\Gamma_0(M))=\bigoplus_{d|M}\gen{g(m\tau):
			g(\tau)\in S_k(\Gamma_0(d))^\new,m|(M/d)}
		$$
		in which every direct summand is invariant under all Hecke operators
		$T_n$, $(n,M)=1$. This implies \eqref{eq: tr(Tn) 1}. Then \eqref{eq:
			tr(Tn) 2} follows from \eqref{eq: tr(Tn) 1} and \eqref{eq:
			inversion}. This proves Part (1). (Note that the case $n=1$ yields
		relations between dimensions. The relations are given as Corollary
		13.3.7 in \cite{Cohen-Stromberg}.)
		
		The proof of Part (2) is easy. We let $\bone$ be the function such
		that $\bone(d)=1$ for all $d\in\N$. Then $\sigma_0=\bone\ast\bone$.
		Since the inverse of $\bone$ for the convolution is $\mu$,
		we have $\delta=\mu\ast\mu$ and hence
		$\sum_{d|M}\delta(d)=(\delta\ast\bone)(M)=\mu(M)$.
		
		We next prove \eqref{eq: psi convolution}. Since $(D,N)=1$, we have
		$$
		\sum_{m|D}\delta(D/m)\psi(mN)
		=\psi(N)(\delta\ast\psi)(D)
		=\psi(N)\prod_{p|D}(\delta\ast\psi)(p).
		$$
		Now $(\delta\ast\psi)(p)=\delta(p)+\psi(p)=-2+(p+1)=p-1$. It follows
		that $(\delta_D\ast\psi)(DN)=\phi(D)\psi(N)$. This proves
		\eqref{eq: psi convolution}.
		
		We now prove \eqref{eq: CM convolution}. According to Lemmas
		\ref{lemma: |CM| 0} and \ref{lemma: |CM| 1}, we have
		\begin{equation*}
			\begin{split}
				|\CM(d;X_0(mN))|&=h(d)\prod_pe(d;\O(1,mN)_p) \\
				&=h(d)\prod_{p|m}\left(1+\left\{\frac dp\right\}\right)
				\prod_{p|N}e(d;\O(1,mN)_p)
			\end{split}
		\end{equation*}
		and
		\begin{equation*}
			\begin{split}
				|\CM(d;X(D,N))|&=h(d)\prod_pe(d;\O(D,N)_p) \\
				&=h(d)\prod_{p|D}\left(1-\left\{\frac dp\right\}\right)
				\prod_{p|N}e(d;\O(D,N)_p),
			\end{split}
		\end{equation*}
		where $e(d;\O_p)=|\Emb(d;\O_p)/\O_p^\times|$ is defined as in Lemma
		\ref{lemma: |CM| 0}. Note that when $p|N$, $\O(1,mN)_p\simeq
		\O(D,N)_p\simeq\O(1,N)_p\simeq\O(1,p^r)_p$, where $p^r$ is the exact
		power of $p$ dividing $N$. Thus, if we define $g$ to be the
		multiplicative function that has value $g(p^r)=e(d;\O(1,p^r)_p)$ at
		prime powers, then the claimed identity \eqref{eq: CM convolution}
		is equivalent to
		\begin{equation} \label{eq: CM tmp}
			(\delta\ast g)(D)=\prod_{p|D}
			\left(1-\left\{\frac dp\right\}\right).
		\end{equation}
		Now for $p|D$, we have $(\delta\ast g)(p)=\delta(p)+g(p)
		=-1+\left\{\frac dp\right\}$. Since $D$ has an even number of prime
		divisors, we see that \eqref{eq: CM tmp} holds. This proves
		\eqref{eq: CM convolution}.
		
		The proof of Part (4) is similar to that of Part (3). We have
		\begin{equation*}
			\begin{split}
				\sum_{m|2D}\delta(2D/m)\psi(mN)
				&=\psi(N)(\delta\ast\psi)(2D)
				=\psi(N)(\delta\ast\psi)(4)\prod_{p|D,p\text{ odd}}
				(\delta\ast\psi)(p).
			\end{split}
		\end{equation*}
		We compute that
		$(\delta\ast\psi)(4)=\psi(4)+\delta(2)\psi(2)+\delta(4)
		=6-2\cdot3+1=1$ and $(\delta\ast\psi)(p)=p-1$. Thus,
		$\sum_{m|2D}\delta(2D/m)\psi(mN)=\phi(D)\psi(N)$. This proves
		\eqref{eq: psi convolution 2}.
		
		For \eqref{eq: CM convolution 2}, we let $g$ be defined as above.
		Then \eqref{eq: CM convolution 2} is equivalent to
		$$
		(\delta\ast g)(2D)=\prod_{p|D,p\text{ odd}}
		\left(1-\left\{\frac dp\right\}\right)\times
		\begin{cases}
			1, &\text{if }d_0\equiv0\mod4\text{ and }2\nmid f, \\
			0, &\text{if }d_0\equiv0\mod4\text{ and }2|f, \\
			\JS{d_0}2, &\text{if }d_0\equiv1\mod4\text{ and }2\nmid f,\\
			-\JS{d_0}2,&\text{if }d_0\equiv1\mod4\text{ and }2\|f, \\
			0, &\text{if }d_0\equiv1\mod4\text{ and }4|f. \end{cases}
		$$
		For an odd prime $p$, we have $(\delta\ast g)(p)=\left\{\frac
		dp\right\}-1$ as before. We then check case by case using
		Lemmas \ref{lemma: |CM| 1} and \ref{lemma: |CM| 2} that
		$$
		(\delta\ast g)(4)=\begin{cases}
			-1, &\text{if }d_0\equiv0\mod4\text{ and }2\nmid f, \\
			0, &\text{if }d_0\equiv0\mod4\text{ and }2|f, \\
			-\JS{d_0}2, &\text{if }d_0\equiv1\mod4\text{ and }2\nmid f,\\
			\JS{d_0}2,&\text{if }d_0\equiv1\mod4\text{ and }2\|f, \\
			0, &\text{if }d_0\equiv1\mod4\text{ and }4|f. \end{cases}
		$$
		Then \eqref{eq: CM convolution 2} follows.
	\end{proof}
	
	We are now ready to prove Theorem \ref{theorem: JL}.
	
	\begin{proof}[Proof of Theorem \ref{theorem: JL}]
		To simplify notations, we will write $S_k(\Gamma_0(M))$ simply as
		$S_k(M)$. By Lemma \ref{lemma: convolution}(1), we have
		$$
		\tr(T_n\big|S_k(M)^\new)
		=\sum_{d|M}\delta(M/d)\tr(T_n\big|S_k(d)).
		$$
		Thus,
		\begin{equation*}
			\begin{split}
				\tr(T_n\big|S_k(DN)^{D\text{-new}})
				&=\sum_{d|DN,D|d}\sigma_0(DN/d)
				\tr(T_n\big|S_k(d)^\new) \\
				&=\sum_{d|DN,D|d}\sigma_0(DN/d)\sum_{m|d}\delta(d/m)
				\tr(T_n\big|S_k(m)) \\
				&=\sum_{m|DN}\tr(T_n\big|S_k(m))
				\sum_{\operatorname{lcm}(m,D)|d,d|DN}\sigma_0(DN/d)\delta(d/m).
			\end{split}
		\end{equation*}
		We now write $m$ as $m=m_1m_2$ with $m_1=(m,D)$ and $m_2=(m,N)$.
		Then setting $d=d'D$, the sum can be written as
		\begin{equation*}
			\begin{split}
				\tr(T_n\big|S_k(DN)^{D\text{-new}})
				&=\sum_{m_1|D}\sum_{m_2|N}\tr(T_n\big|S_k(m_1m_2))
				\sum_{m_2|d',d'|N}\sigma_0(N/d')
				\delta(d'D/m_1m_2) \\
				&=\sum_{m_1|D}\delta(D/m_1)
				\sum_{m_2|N}\tr(T_n\big|S_k(m_1m_2)) \\
				&\qquad\qquad\times\sum_{d''|(N/m_2)}
				\sigma_0(N/d''m_2)\delta(d'').
			\end{split}
		\end{equation*}
		Applying \eqref{eq: delta} to the innermost sum, we obtain
		$$
		\sum_{d''|(N/m_2)}\sigma_0(N/d''m_2)\delta(d'')
		=\begin{cases}
			1, &\text{if }m_2=N, \\
			0, &\text{else}. \end{cases}
		$$
		It follows that
		\begin{equation} \label{eq: RHS 1}
			\tr(T_n\big|S_k(DN)^{D\text{-new}})
			=\sum_{m|D}\delta(D/m)\tr(T_n\big|S_k(mN))
		\end{equation}
		Similarly, we have
		\begin{equation} \label{eq: RHS 2}
			\tr(T_n\big|S_k(2DN)^{2D\text{-new}})
			=\sum_{m|2D}\delta(2D/m)\tr(T_n\big|S_k(mN)).
		\end{equation}
		We now write $\tr(T_n\big|S_k(M))$ as
		$$
		\tr(T_n\big|S_k(M))=\frac{k-1}{12}\alpha_nA_1(M)
		-\frac12A_2(M)-A_3(M)+\beta_kA_4(M)\sum_{t|n}t
		$$
		according to Lemma \ref{lemma: trace Gamma0(M)}, where
		$$
		A_1(M)=\psi(M),\qquad
		A_2(M)=\sum_{\substack{t\in\Z\\t^2<4n}}
		\frac{\rho_{t,n}^{k-1}-\overline\rho_{t,n}^{k-1}}
		{\rho_{t,n}-\overline\rho_{t,n}}
		\sum_{r^2d=t^2-4n}\frac1{w_d}|\CM(d;X_0(M))|,
		$$
		$$
		A_3(M)={\sum_{\substack{d|n\\d\le\sqrt n}}}'d^{k-1}
		\sum_{\substack{c|M\\(c,M/c)|(M,n/d-d)}}\phi((c,M/c)),
		\qquad A_4(M)=1.
		$$
		Then
		\begin{equation*}
			\begin{split}
				\tr(T_n\big|S_k(DN)^{D\text{-new}})
				&=\sum_{m|D}\delta(D/m)\Bigg(\frac{k-1}{12}\alpha_nA_1(mN) \\
				&\qquad
				-\frac12A_2(mN)-A_3(mN)+\beta_kA_4(mN)\sum_{t|n}t\Bigg).
			\end{split}
		\end{equation*}
		By Lemma \ref{lemma: convolution}, we have
		$$
		\sum_{m|D}\delta(D/m)A_1(mN)=\phi(D)\psi(N),
		$$
		$$
		\sum_{m|D}\delta(D/m)A_2(mN)=\sum_{\substack{t\in\Z\\t^2<4n}}
		\frac{\rho_{t,n}^{k-1}-\overline\rho_{t,n}^{k-1}}
		{\rho_{t,n}-\overline\rho_{t,n}}
		\sum_{r^2d=t^2-4n}\frac{h(d)}{w_d}\prod_{p|DN}e_p(d),
		$$
		where $e_p(d)$ are defined by \eqref{eq: epsilon d}, and  
		$$
		\sum_{m|D}\delta(D/m)A_4(mN)=\mu(D)=1.
		$$
		For the sum involving $A_3(mN)$, we note that the inner sum in
		$A_3(mN)$ is equal to $\sum_{c|mN}1=\sigma_0(mN)$. Thus, we have
		$$
		\sum_{m|D}\delta(D/m)A_3(mN)
		=C\sum_{m|D}\delta(D/m)\sigma_0(mN)
		=C\sigma_0(N)(\delta\ast\sigma_0)(D),
		$$
		where
		$$
		C={\sum_{\substack{d|n\\d\le\sqrt n}}}'d^{k-1},
		$$
		By \eqref{eq: delta}, $(\delta\ast\sigma_0)(D)=0$. Therefore,
		we have
		$$
		\sum_{m|D}\delta(D/m)A_3(mN)=0.
		$$
		Combining everything, we obtain
		\begin{equation} \label{eq: LHS first}
			\begin{split}
				\tr(T_n\big|S_k(DN)^{D\text{-new}})
				&=\frac{k-1}{12}\alpha_n\phi(D)\psi(N) \\
				&-\frac12\sum_{\substack{t\in\Z\\t^2<4n}}
				\frac{\rho_{t,n}^{k-1}-\overline\rho_{t,n}^{k-1}}
				{\rho_{t,n}-\overline\rho_{t,n}}
				\sum_{r^2d=t^2-4n}\frac{h_d}{w_d}\prod_{p|DN}e_p(d)
				+\beta_k\sum_{t|n}t.
			\end{split}
		\end{equation}
		(Note that this reproves the Jacquet-Langlands correspondence
		between $S_k(\Gamma_0(DN))^{D\text{-new}}$ and $S_k(\Gamma(D,N))$.)
		
		The trace of $T_n\big|S_k(2DN)^{2D\text{-new}}$ is computed in
		the same way. We have
		\begin{equation*}
			\begin{split}
				\tr(T_n\big|S_k(2DN)^{2D\text{-new}})
				&=\sum_{m|2D}\delta(2D/m)\Bigg(\frac{k-1}{12}\alpha_nA_1(mN) \\
				&\qquad
				-\frac12A_2(mN)-A_3(mN)+\beta_kA_4(mN)\sum_{t|n}t\Bigg).
			\end{split}
		\end{equation*}
		Applying Lemma \eqref{lemma: convolution}, we find that
		$$
		\sum_{m|2D}\delta(2D/m)A_1(mN)=\phi(D)\psi(N),
		$$
		and
		$$
		\sum_{m|2D}\delta(2D/m)A_2(mN)=\sum_{\substack{t\in\Z\\t^2<4n}}
		\frac{\rho_{t,n}^{k-1}-\overline\rho_{t,n}^{k-1}}
		{\rho_{t,n}-\overline\rho_{t,n}}
		\sum_{r^2d=t^2-4n}\frac{h(d)}{w_d}\widetilde e_2(d)
		\prod_{p|DN,p\neq2}e_p(d),
		$$
		where $\widetilde e_2(d)$ is defined by \eqref{eq: e'}.  
		The sums involving $A_3(mN)$ and $A_4(mN)$ are a bit different from
		the case of $S_k(DN)^{D\text{-new}}$. Consider the inner sum
		$$
		\sum_{\substack{c|mN\\(c,mN/c)|(mN,n/d-d)}}\phi((c,mN/c))
		$$
		in $A_3(mN)$. Write $c$ as $c=c_1c_2$ with $c_1|2D$ and $c_2|N$.
		Then the inner sum is equal to
		\begin{equation} \label{eq: A3 tmp}
			\sum_{\substack{c_2|N\\(c_2,N/c_2)|(N,n/d-d)}}\phi((c_2,N/c_2))
			\sum_{\substack{c_1|m\\(c_1,m/c_1)|(m,n/d-d)}}
			\phi((c_1,m/c_1)).
		\end{equation}
		Observe that for $c_1|m$, $(c_1,m/c_1)$ is either $1$ or $2$, so
		$\phi((c_1,m/c_1))$ is always $1$. Furthermore, $(c_1,m/c_1)=2$
		occurs only when $4|m$ and $2\|c_1$. Since $n$ is odd, the 
		integer $n/d-d$ is always even. Thus, the condition
		$(c_1,m/c_1)|(m,n/d-d)$ holds for any divisor $c_1$ of $m$.
		Therefore, the sum in \eqref{eq: A3 tmp} is reduced to
		$$
		\sigma_0(m)\sum_{\substack{c_2|N\\(c_2,N/c_2)|(N,n/d-d)}}
		\phi((c_2,N/c_2)).
		$$
		Then since
		$\sum_{m|2D}\delta(2D/m)\sigma_0(m)=(\delta\ast\sigma_0)(2D)=0$,
		we find that
		$$
		\sum_{m|2D}\delta(2D/m)A_3(mN)=0.
		$$
		For the sum involving $A_4(mN)$, we have, by Lemma \ref{lemma:
			convolution}(2),
		$$
		\sum_{m|2D}A_4(mN)=\mu(2D)=0.
		$$
		Altogether, we see that
		\begin{equation*}
			\begin{split}
				\tr(T_n\big|S_k(2DN)^{2D\text{-new}})
				&=\frac{k-1}{12}\alpha_n\phi(D)\psi(N) \\
				&-\frac12\sum_{\substack{t\in\Z\\t^2<4n}}
				\frac{\rho_{t,n}^{k-1}-\overline\rho_{t,n}^{k-1}}
				{\rho_{t,n}-\overline\rho_{t,n}}
				\sum_{r^2d=t^2-4n}\frac{h(d)}{w_d}\widetilde e_2(d)
				\prod_{\substack{p|DN\\p\neq2}}e_p(d).
			\end{split}
		\end{equation*}
		Combining this with \eqref{eq: LHS first}, we obtain
		\begin{equation*}
			\begin{split}
				&\tr(T_n\big|S_k(DN)^{D\text{-new}})
				+2 \tr(T_n\big|S_k(2DN)^{2D\text{-new}}) \\
				&\qquad=\frac{k-1}{4}\alpha_n\phi(D)\psi(N) \\
				&\qquad-\frac12\sum_{\substack{t\in\Z\\t^2<4n}}
				\frac{\rho_{t,n}^{k-1}-\overline\rho_{t,n}^{k-1}}
				{\rho_{t,n}-\overline\rho_{t,n}}
				\sum_{r^2d=t^2-4n}\frac{h(d)}{w_d}(e_2(d)+2\widetilde e_2(d))
				\prod_{\substack{p|DN\\p\neq2}}e_p(d)+\beta_k\sum_{t|n}t.
			\end{split}
		\end{equation*}
		On the other hand, by Proposition \ref{proposition:
			trace formula}, we have
		\begin{equation*}
			\begin{split}
				\tr(T_n\big|\Gamma'(D,N))
				&=\frac{k-1}{4}\alpha_n\phi(D)\psi(N) \\
				&-\frac12\sum_{\substack{t\in\Z\\t^2<4n}}
				\frac{\rho_{t,n}^{k-1}-\overline\rho_{t,n}^{k-1}}
				{\rho_{t,n}-\overline\rho_{t,n}}
				\sum_{\substack{r^2d=t^2-4n\\4|d}}\frac1{w_d}|\CM(d;X'(D,N))|
				+\beta_k\sum_{t|n}t.
			\end{split}
		\end{equation*}
		Comparing the two expressions, we see that to prove the theorem, it
		suffices to show that for all integers $t$ such that $t^2<4n$, one
		has
		\begin{equation} \label{eq: goal}
			\sum_{r^2d=t^2-4n}\frac{h(d)}{w_d}(e_2(d)+2\widetilde e_2(d))
			\prod_{\substack{p|DN\\p\neq2}}e_p(d)
			=\sum_{\substack{r^2d=t^2-4n\\4|d}}\frac1{w_d}|\CM(d;X'(D,N))|.
		\end{equation}
		
		Let us first consider the case $t$ is odd. In this case, the sum in
		the right-hand side of \eqref{eq: goal} is empty. On the other
		hand, since $n$ is odd, the discriminant $d$ in the sum is always
		congruent to $5$ modulo $8$. Consequently, we have, by \eqref{eq:
			e'},
		$$
		e_2(d)+2\widetilde e_2(d)=2+2\times(-1)=0.
		$$
		Thus, the left-hand side of \eqref{eq: goal} is also equal to $0$.
		This proves \eqref{eq: goal} for the case $t$ is odd.
		
		From now on we assume that $t$ is even. Let $d_0$ be the
		discriminant of the field $\Q(\sqrt{t^2-4n})$ and for $d$ such that
		$r^2d=t^2-4n$ for some $r$, we write $d$ as $d=f^2d_0$. Consider the
		case $4|d_0$. According to \eqref{eq: e'},
		$$
		e_2(d)=\widetilde e_2(d)=\begin{cases}
			1, &\text{if }2\nmid f, \\
			0, &\text{if }2|f. \end{cases}
		$$
		Either way, we find that $e_2(d)+2\widetilde e_2(d)=3e_2(d)$ and
		\eqref{eq: goal}. On the other hand, since $4|d_0$, by Lemma
		\ref{lemma: CM on X'} and the definition \eqref{eq: epsilon d} of
		$e_p(d)$,
		$$
		|\CM(d;X'(D,N))|=3|\CM(d;X(D,N))|
		=3h(d)\prod_{p|DN}e_p(d).
		$$
		Thus, \eqref{eq: goal} holds in the case $4|d_0$.
		
		We next consider the case $d_0\equiv 1\mod 8$. In this case,
		we have
		$$
		e_2(d)+2\widetilde e_2(d)=\begin{cases}
			0+2=2, &\text{if }2\nmid f,\\
			0-2=-2, &\text{if }2\|f, \\
			0, &\text{if }4|f. \end{cases}
		$$
		Therefore, the left-hand side of \eqref{eq: goal} is equal to
		$$
		2\sum_{d:2\nmid f}
		h(d)\prod_{p|DN,p\neq2}e_p(d)
		-2\sum_{d:2\|f}
		h(d)\prod_{p|DN,p\neq2}e_p(d).
		$$
		Now recall that if $d$ is a discriminant such that $d\equiv1\mod 8$,
		then $h(4d)=h(d)$. Thus, the two sums above actually cancel out and
		the left-hand side of \eqref{eq: goal} is equal to $0$. On the other
		hand, the right-hand side of \eqref{eq: goal} is also equal to $0$
		due to the fact that an imaginary quadratic number field of
		discriminant congruent to $1$ modulo $8$ cannot be embedded into
		$\B_D$. We conclude that \eqref{eq: goal} holds when
		$d_0\equiv1\mod8$.
		
		We now consider the last case $d_0\equiv5\mod 8$. We have
		$$
		e_2(d)+2\widetilde e_2(d)=\begin{cases}
			2-2=0, &\text{if }2\nmid f, \\
			0+2=2, &\text{if }2\|f, \\
			0, &\text{if }4|f. \end{cases}
		$$
		Thus, the left-hand side of \eqref{eq: goal} is equal to
		$$
		2\sum_{d:2\|f}\frac{h(d)}{w_d}\prod_{p|DN,p\neq2}e_p(d).
		$$
		Recall the fact that if a discriminant $d$ is congruent to $5$
		modulo $8$, then $h(4d)=3h(d)/w_d$. Therefore, the sum above is
		equal to
		$$
		6\sum_{d\equiv 5\mod 8}\frac{h(d)}{w_d}\prod_{p|DN,p\neq2}e_p(d).
		$$
		On the other hand, by Lemma \ref{lemma: CM on X'}, the right-hand
		side of \eqref{eq: goal} is equal to
		\begin{equation*}
			\begin{split}
				\sum_{4\|d}|\CM(d;X'(D,N))|
				&=3\sum_{d\equiv5\mod 8}\frac{1}{w_d}|\CM(d;X(D,N))| \\
				&=6\sum_{d\equiv5\mod 8}\frac{h(d)}{w_d}\prod_{p|DN,p\neq2}
				e_p(d).
			\end{split}
		\end{equation*}
		Therefore, \eqref{eq: goal} holds for the case $d_0\equiv5\mod 8$ as
		well. This completes the proof of the theorem.
	\end{proof}
	
	\section{Proof of Theorem \ref{theorem: JL1}}
	
	In this section, we will prove Theorem \ref{theorem: JL1}.
	To prove the theorem, we first introduce  an isomorphism from $S_k(\Gamma'(D,N),\chi)$ to $S_k(\Gamma'(D,N),\overline\chi)$ that is the analogue of the map $f\to f^c$ in the setting of classical modular forms, where $f^c(\tau):=\overline{f(-\overline\tau)}$. Then we will show that Hecke operators  on $S_k(\Gamma'(D,N))$ are self-adjoint with respect to the Petersson inner product, and hence their eigenvalues are real.

	By Lemma \ref{lemma: Gamma'(D,N)}(4), $\O'(D,N)$ has an element $\sigma$ of reduced norm $-1$. For $f\in S_k(\Gamma'(D,N))$, define $f^c$ by
	$$
	f^c(\tau):=\overline{(f\big|_k\sigma)(\overline\tau)}
	=\frac1{(c\tau+d)^k}\overline{f(\sigma\overline\tau)},
	$$
	where we write $\iota(\sigma)=\SM abcd$.
	It is easy to check that the definition of $f^c$ does not depend on the choice of $\sigma$, as $[\O'(D,N)^\times:\G'(D,N)]=2$. The linear map $f\mapsto f^c$ has the following properties.
	
	\begin{Lemma} \label{lemma: c}
		\begin{enumerate}
			\item We have $(f^c)^c=f$, i.e., $f\mapsto f^c$ is an involution on $S_k(\Gamma'(D,N))$.
			\item For a positive integer $n$ relatively prime to $DN$, we have $T_n\circ c=c\circ T_n$, i.e., the involution $f\mapsto f^c$ commutes with Hecke operators $T_n$.
			\item 
			For a character $\chi$ of $\Gamma(D,N)/\Gamma'(D,N)$, the map $f\mapsto f^c$ is an isomorphism from $S_k(\Gamma'(D,N),\chi)$ to $S_k(\Gamma'(D,N),\overline\chi)$.
		\end{enumerate}
	\end{Lemma}
	
	\begin{proof}
		Let $\sigma$ be an element of reduced norm $-1$ in $\O'(D,N)$ that defines $f^c$. Part (1) follows from the fact that $\sigma^2\in\Gamma'(D,N)$.
		Also, Part (2) follows from the fact that $\sigma$ normalizes both $M(n)$ and $\Gamma'(D,N)$.
		
		We now prove Part (3). Let $f\in S_k(\Gamma'(D,N),\chi)$.
		For $\gamma\in\Gamma(D,N)$, let $\gamma'=\sigma\gamma\sigma^{-1}$. By Lemma \ref{lemma: Gamma'(D,N)}(3), $\gamma$ and $\gamma'$ are in the same coset of $\Gamma'(D,N)$ in $\Gamma(D,N)$. Therefore,
		$$
		(f^c\big|_k\gamma)(\tau)=\overline{(f\big|_k\sigma\gamma)(\overline\tau)}=\overline{(f\big|_k\gamma'\sigma)(\overline\tau)}=\overline{\chi(\gamma')
			(f\big|_k\sigma)(\overline\tau)}
		=\overline{\chi(\gamma)}f^c(\tau).
		$$
		This shows that the involution $f\mapsto f^c$ maps  $S_k(\Gamma'(D,N),\chi)$ to $S_k(\Gamma'(D,N),\overline\chi)$ and defines an isomorphism between the two spaces.
	\end{proof}
	
	For two modular forms $f$ and $g$ on a subgroup $\Gamma$ of finite index of $\Gamma(D,N)$, we let
	$$
	\gen{f,g}:=
	\frac1{\operatorname{vol}(\Gamma\backslash\H)}
	\int_{\Gamma\backslash\H}
	f(\tau)\overline{g(\tau)}y^k\frac{dx\,dy}{y^2}
	$$
	be the Petersson inner product. We now show that the Hecke operators are Hermitian. 
	\begin{Lemma} \label{lemma: self-adjoint}
		Assume that $n$ is a positive integer relatively prime to $DN$. Then the Hecke operator $T_n$ on $S_k(\Gamma'(D,N))$ is self-adjoint with respect to the Petersson inner product. Consequently, every eigenvalue of $T_n$ is real.
	\end{Lemma}
	
	\begin{proof} 
		Since $S_k(\Gamma'(D,N))=\oplus_\chi S_k(\Gamma'(D,N),\chi)$ and each $S_k(\Gamma'(D,N),\chi)$ is Hecke-invariant, where $\chi$ are characters of $\Gamma(D,N)/\Gamma'(D,N)$, it suffices to prove that Hecke operators are self-adjoint on each $S_k(\Gamma'(D,N),\chi)$. Moreover, since the Hecke algebra is generated by $T_p$ for primes $p$ not dividing $DN$, we only need to prove that $T_n$ is self-adjoint on $S_k(\Gamma'(D,N),\chi)$ for the case $n$ is a prime.
		
		We first prove that if $\gamma_1$ and $\gamma_2$ are two elements of reduced norm $p$ in $\O'(D,N)$, then for $f,g\in S_k(\Gamma'(D,N),\chi)$ we have 
		\begin{equation} \label{eq: inner products equal} \gen{f\big|_k\gamma_1,g}=\gen{f\big|_k\gamma_2,g}.
		\end{equation}
		(Note that $f\big|_k\gamma_j$ is a modular form on some subgroup of finite index of $\Gamma'(D,N)$.)
		Indeed, by Lemma \ref{lemma: Gamma'(D,N)}(5), there are elements $\alpha$ and $\beta$ of $\Gamma(D,N)$ with $\alpha\beta\in\Gamma'(D,N)$ such that $\gamma_1=\alpha\gamma_2\beta$. Then the standard properties of the Petersson inner product imply that
		\begin{equation*}
			\begin{split}
				\gen{f\big|_k\gamma_1,g}
				&=\gen{f\big|_k\alpha\gamma_2\beta,g}    =\chi(\alpha)\gen{f\big|_k\gamma_2\beta,g}
				=\chi(\alpha)\gen{f\big|_k\gamma_2,g\big|_k\beta^{-1}}\\
				&=\chi(\alpha)\chi(\beta)\gen{f\big|_k\gamma_2,g}
				=\gen{f\big|_k\gamma_2,g}.
			\end{split}
		\end{equation*}
		This proves \eqref{eq: inner products equal}. Consequently,
		we have
		$$
		\gen{T_pf,g}=(p+1)\gen{f\big|_k\gamma,g}, \qquad
		\gen{f,T_pg}=(p+1)\gen{f,g\big|_k\gamma}
		$$
		for any element $\gamma$ of reduced norm $p$ in $\O'(D,N)$. Here $p+1=|\G'(D,N)\backslash M(p)|$. 
		Now we have $\gen{f\big|_k\gamma,g}=\gen{f,g\big|_k\overline\gamma}$, where $\overline \g \g=nI$.
		Since $\gamma$ and $\overline\gamma$ are both elements of reduced norm $p$ in $\O'(D,N)$, by \eqref{eq: inner products equal}, we have $\gen{f,g\big|_k\gamma}=\gen{f,g\big|_k\overline\gamma}$.
		It follows that $T_p$ is self-adjoint on $S_k(\Gamma'(D,N),\chi)$ and the proof of the lemma is complete.
	\end{proof}
	
	We are now ready to prove Theorem \ref{theorem: JL1}.
	
	\begin{proof}[Proof of Theorem \ref{theorem: JL1}]
		Let $\chi$ be a nontrivial character of $\Gamma(D,N)/\Gamma'(D,N)$. By Lemma \ref{lemma: self-adjoint}, Hecke operators $T_n$, $(n,DN)=1$, are commuting self-adjoint linear operators on $S_k(\Gamma'(D,N),\chi)$ and $S_k(\Gamma'(D,N),\overline\chi)$. Thus, the two spaces of modular forms have bases consisting of simultaneous eigenforms for all Hecke operators.
		Moreover, by Lemma \ref{lemma: c}, if $f$ is a Hecke eigenform in $S_k(\Gamma'(D,N),\chi)$, then $f^c$ is a Hecke eigenform in $S_k(\Gamma'(D,N),\overline\chi)$ and the eigenvalues are related by $T_nf^c=\overline{\lambda_n(f)}f^c$, where $\lambda_n(f)$ is the eigenvalue of $T_n$ corresponding to $f$.
		Now by Lemma \ref{lemma: self-adjoint}, all eigenvalues $\lambda_n(f)$ are real. Therefore, the eigenvalue of $T_n$ corresponding to $f^c$ is the same as that corresponding to $f$. It follows that
		$$
		\tr(T_n|S_k(\Gamma'(D,N),\chi))
		=\tr(T_n|S_k(\Gamma'(D,N),\overline\chi).
		$$
		Finally, by the classical Jacquet-Langlands correspondence for Eichler orders and Theorem \ref{theorem: JL}, we have
		$$
		\tr(T_n|S_k(\Gamma'(D,N)),\chi)
		+\tr(T_n|S_k(\Gamma'(D,N),\overline\chi)
		=2\tr(T_n|S_k(\Gamma_0(2DN))^{2D\text{-new}}).
		$$
		From this, we conclude that
		$$
		\tr(T_n|S_k(\Gamma'(D,N),\chi))
		=\tr(T_n|S_k(\Gamma_0(2DN))^{2D\text{-new}}).
		$$
		This completes the proof of Theorem \ref{theorem: JL1}.
	\end{proof}
	\bibliographystyle{plain}
	\bibliography{Jacquet-Langlands}
\end{document}